\documentclass[11pt,reqno]{amsart}
\usepackage{fullpage,enumerate}
\usepackage{amsthm,amsfonts,amsmath,amssymb}
\usepackage[utf8]{inputenc}
\usepackage{bm}
\usepackage{mathtools}%floor
\usepackage{enumitem}
\usepackage{listings}
\usepackage{mathrsfs}
\usepackage{fancybox}
\usepackage{dsfont}
\usepackage{float}
\usepackage{arcs}
\usepackage{pgf,tikz}
\usetikzlibrary{arrows}
\usepackage{graphicx}
\usepackage{color}
\usepackage[all]{xy}

\usepackage{tikz}
   \usetikzlibrary{calc}

\numberwithin{equation}{section}

\newtheorem{thm}{Theorem}[section]
\newtheorem{cor}[thm]{Corollary}
\newtheorem{lem}[thm]{Lemma}
\newtheorem{prop}[thm]{Proposition}

\newtheorem{theorem}{Theorem}

\theoremstyle{definition}

\newtheorem{defn}[thm]{Definition}

\theoremstyle{remark}
\newtheorem{rem}[thm]{Remark}

\newcommand{\C}{\mathbb{C}}
\newcommand{\R}{\mathbb{R}}
\newcommand{\Z}{\mathbb{Z}}

\newcommand{\cp}[1]{\C P^{#1}}

\newcommand{\LL}{\vert\mathcal{L}\vert}
\newcommand{\Li}{\mathcal{L}}

\definecolor{bluegray}{rgb}{0.4, 0.6, 0.8}
\renewcommand{\emph}[1]{{\it {\color{bluegray} #1}}}

\newcommand{\A}{\mathcal{A}}
\newcommand{\ttor}{(\C^\ast)^2}

\newcommand{\D}{\mathcal D}

\newcommand{\Xb}{X^\bullet}
\newcommand{\Cb}{C^\bullet}
\newcommand{\tCb}{\widetilde C^\bullet}
\newcommand{\tC}{\widetilde C}

\newcommand{\sv}[1]{V_{#1}}
\newcommand{\irr}{\mathrm{irr}}
\newcommand{\trop}{\mathrm{tr}}
\renewcommand{\:}{\colon}
\newcommand{\pkk}{{\mathcal{P}}^{k',k}}

\DeclareMathOperator{\ord}{\textit{ord}}

\DeclareMathOperator{\conv}{conv}

\DeclareMathOperator{\conj}{conj}

\DeclareMathOperator{\New}{New}

\DeclareMathOperator{\Area}{Area}

\title{A note on the Severi problem for toric surfaces}
\author{Lionel Lang and Ilya Tyomkin}

\thanks{IT is partially supported by the Israel Science Foundation (grant No. 821/16).}

\address[Lang]{Department of Mathematics\\
	Stockholm University\\SE - 106 91\\ Stockholm\\ Sweden}
\email{lang@math.su.se}

\address[Tyomkin]{Department of Mathematics\\
	Ben-Gurion University of the Negev\\P.O.Box 653 \\Be'er Sheva\\ 84105\\  Israel }\email{tyomkin@math.bgu.ac.il}

\keywords{Severi variety, Severi problem, toric surfaces, simple Harnack curves.}
\subjclass[2010]{14H10, 14M25, 14T90}

\begin{document}

 \title{A note on the Severi problem for toric surfaces}

\begin{abstract}
In this note, we make a step towards the classification of toric surfaces admitting reducible Severi varieties. We generalize the results of \cite{L19, Tyom13, Tyom14}, and provide two families of toric surfaces admitting reducible Severi varieties. The first family is general, and is obtained by a quotient construction. The second family is exceptional, and corresponds to certain narrow polygons, which we call kites. We introduce two types of invariants that distinguish between the components of the Severi varieties, and allow us to provide lower bounds on the numbers of the components. The sharpness of the bounds is verified in some cases, and is expected to hold in general for ample enough linear systems. In the appendix, we establish a connection between the Severi problem and the topological classification of univariate polynomials.
\end{abstract}
	
\maketitle
	
\setcounter{tocdepth}{1}
\tableofcontents

\section{Introduction}

The study of families of curves on algebraic surfaces is a classical problem in algebraic geometry. In 1920's, Severi considered the locus $V^{\irr}_{g,d}$ of integral, degree $d$, genus $g$, planar curves in an attempt to prove the irreducibility of the moduli spaces of algebraic curves. Such loci for different surfaces and linear systems are now called {\em Severi varieties}. In 1986, Harris proved the irreducibility of the classical Severi varieties $V^{\irr}_{g,d}$ in characteristic zero \cite{Ha}, and very recently Christ, He, and the second author found a characteristic-free proof of Harris' theorem \cite{CHT20}. Over the years, the irreducibility of Severi varieties was proved for other surfaces, such as Hirzebruch surfaces, Del-Pezzo surfaces, and certain toric and K3 surfaces, see, e.g., \cite{Bou, CiDe, Tes, Tyom}. Most of the results apply only in the case of small genus, but some are general.

On the negative side, for many surfaces of general type, Severi varieties are known to be reducible, and even non-equidimensional, see \cite{CC99}. In 2013, the second author found first examples of reducible Severi varieties on toric surfaces, initially in positive characteristic \cite{Tyom13}, and then in characteristic zero \cite{Tyom14}. A different type of examples was discovered recently by the first author in his study of the monodromy action on the set of nodes of rational curves on toric surfaces \cite{L19}.

The goal of the current paper is to make a step towards a classification of toric surfaces admitting reducible Severi varieties, and towards a description of the corresponding irreducible components. We unify the examples of \cite{L19, Tyom13, Tyom14}, and introduce certain invariants that distinguish between components of Severi varieties on toric surfaces. As a result, we are able to provide lower bounds on the numbers of the irreducible components.

\subsection{The main results}\label{sec:mainres}

Let $g\geqslant 0$ be an integer, $\Delta\subset \R^2$ a lattice polygon, and $(X,\Li)$ the corresponding polarized toric surface. Recall that the Severi varieties $V_{0,\Li}^\irr$ are irreducible by \cite[Proposition~4.1]{Tyom}. The following theorem provides a general lower bound on the number of irreducible components of the Severi variety $V_{g,\Li}^\irr$ in the case of positive genus.

\begin{theorem}\label{thm:general}
If $g\geqslant 1$, then the number of irreducible components of the Severi variety $V_{g,\Li}^\irr$ is bounded from below by the number of affine sublattices $M\subseteq\Z^2$ for which the following two conditions hold: (a) $\partial\Delta\cap M=\partial\Delta\cap\Z^2$, and (b) $|\Delta^\circ\cap M|\geqslant g$; where $\partial\Delta$ and $\Delta^\circ$ denote the boundary and the interior of $\Delta$, respectively.
\end{theorem}

For any sublattice $M\subset\Z^2$, we set $\delta_M(\Delta,g):=|\Delta^\circ\cap M|-g$. Then condition (b) reads: $\delta_M(\Delta,g)\geqslant 0$.
We now restrict our attention to a special family of lattice polygons that we call kites. Let $k',k$ be non-negative integers such that $k'\geqslant k$ and $k'>0$. A {\em kite $\Delta_{k,k'}$} is the polygon with vertices $(0,0),(\pm 1,k), (0,k+k')$. Let $\Delta=\Delta_{k,k'}$ be a kite, and $g\geqslant 0$ an integer. We define \emph{$\#_{k,k',g}$} to be the number of sublattices $M\subseteq \Z^2$ satisfying the conditions (a) and (b) of Theorem~\ref{thm:general} and counted with the following multiplicities: if the index of $M$ in $\Z^2$ is odd, then the multiplicity is equal to the number of integers $0\leqslant \kappa\leqslant \min\{\delta_M(\Delta,g), g\}$ that are congruent to $\delta_M(\Delta,g)$ modulo $2$, and if the index is even, then the multiplicity is one.

\begin{theorem}\label{thm:kite}
If $\Delta=\Delta_{k,k'}$ is a kite, and $(X,\Li)$ the corresponding polarized toric surface, then for any $g\geqslant 1$, the number of irreducible components of $V_{g,\Li}^\irr$ is bounded from below by $\#_{k,k',g}$.
\end{theorem}

We provide two proofs of the main results. One is tropical and another is topological. The topological approach is based on the study of deformations of simple Harnak curves, and one of its key ingredients is the following generalization of a result of Kenyon and Okounkov \cite[Proposition 10]{KO}, which we believe to be of an independent interest.

Let $[C]\in V_{g,\Li}^\irr$ be a curve, and $\Upsilon:=\{\gamma_1,\dotsc,\gamma_g\}$ a collection of  oriented simple closed curves in the smooth locus of $C$ that are contained in the open torus orbit $\Xb\subset X$ and are contractible in $\Xb$. Since $\Xb$ is acyclic, each $\gamma_j$ bounds a smooth disc $M_j\subset \Xb$ that is unique up to homotopy. Set $\Omega:= \frac{dz\wedge dw}{zw}$, where $(z,w)$ are toric coordinates on $\Xb$, and consider the integrals
\begin{equation}\label{eq:integrals}
\int_{M_j} \Omega \; , \quad j=1,\dotsc,g.
\end{equation}
Since $\Omega$ is closed, these integrals are independent of the choice of $M_j$'s. Furthermore, since $\Omega$ vanishes on $C$, the integrals are also independent of the choice of $\gamma_j$'s within their isotopy classes. Finally, any small deformation $C_t$ of $C$ induces a small deformation $\Upsilon_t$ of $\Upsilon$. Although the later is not uniquely defined, the isotopy classes of $\gamma_{j,t}$'s are. Therefore, integrals \eqref{eq:integrals}, induce well-defined functions on $V_{g,\Li}^\irr$ in a neighborhood of $[C]$.

\begin{theorem}\label{thm:integrals}
Let $W\subset V_{g,\Li}^\irr$ be the locus of curves having a given tangency profile with the boundary divisor, $[C]\in W$ any curve, and $\Upsilon$ a collection of loops as above. If $\Upsilon$ generates a $g$-dimensional subspace of $H_1(C,\Z)$, then the integrals \eqref{eq:integrals} together with the coordinates of the intersection points of $C$ with the boundary divisor of $X$ provide a local system of coordinates on $W$.
\end{theorem}

We refer to \cite[Lemma 3.1]{IM} for a similar statement in the case of curves on K3-surfaces.

\subsection{Discussion}\label{sec:tandp}

To prove the irreducibility, or more generally, to classify the irreducible components of a Severi variety, one often follows the following strategy. First, one proves that the closure of any component contains components of Severi varieties of smaller genera, and then uses monodromy type argument to get an upper bound on the number of components.

The question, whether the closure of any component of a Severi variety on a toric surface necessarily contains components of Severi varieties of smaller genera, is an interesting open problem. This is known to be the case for $\mathbb P^2$ by \cite{Ha}, for Hirzebruch surfaces by \cite{Tyom}, and most generally, for toric surfaces associated to h-transverse polygons by \cite{CHT20a}. We expect this to be true in general, but we were not able to prove this so far. Let us denote the union of the irreducible components of $V_{g,\Li}^\irr$ containing $V_{0,\Li}^\irr$ in their closure by $V'_{g,\Li}$.

A monodromy argument allows one to provide an upper bound on the number of components of $V'_{g,\Li}$. Indeed, it is not difficult to verify that in a neighborhood of an integral, rational, nodal curve $C_0$, the closure of $V'_{g,\Li}$ consists of smooth branches, parameterized by subsets of $g$ nodes of $C_0$. Therefore, the number of such components is bounded from above by the number of orbits of sets of $g$ nodes under the monodromy action on the set of nodes of $C_0$. For planar curves, the monodromy acts as the full symmetric group by \cite{Ha}. More generally, this is the case if the toric surface is smooth at (at least) one of its zero-dimensional orbits, and the polarization is ample enough; see \cite{L19} for details. Thus, the upper bound in such cases is one.

To get a lower bound on the number of components of a Severi variety, one must find invariants that can distinguish between the components. This is precisely what we do in the current paper. Unfortunately, even if $V'_{g,\Li}=V_{g,\Li}^\irr$, there is often a discrepancy between the upper bound prescribed by the monodromy calculations of \cite{L19}, and the lower bounds coming from the admissible values of the invariants we construct.

One could naively guess that the monodromy action on the nodes of $C_0$ should prescribe the number of components of $V'_{g,\Li}$. However, this turns out to be wrong as the following example shows. Let $\Delta$ be the triangle with vertices $(0,0),(4,1),(0,3)$, and $(X,\Li)$ the corresponding polarized toric surface. Consider the Severi variety $V_{2,\Li}^\irr$. Using the techniques of \cite{CHT20a} and \cite{CHT20}, one can show that any component of $V_{2,\Li}^\irr$ contains $V_{0,\Li}^\irr$ in its closure, and that such component is unique. However, the number of orbits of pairs of nodes of $C_0$ under the monodromy action is two. Therefore, it is not sufficient to analyze the monodromy action on the sets of nodes of rational curves. We hope that investigating the monodromy actions on the nodes of curves of higher genera could close the gap between the lower and the upper bounds in most cases.

Next let us discuss the sharpness of the lower bounds in our main results, and speculate on the expected number of the irreducible components of Severi varieties on toric surfaces in general. First, notice that Theorem~\ref{thm:kite} implies that the bound of Theorem~\ref{thm:general} is not sharp for some special toric surfaces. However, we expect the bound to be equal to the actual number of irreducible components of $V_{g,\Li}^\irr$ at least if the line bundle $\Li$ is ample enough. In the case of kites, we have $V'_{g,\Li}=V_{g,\Li}^\irr$ by \cite{CHT20a}. However, there is a discrepancy between the upper and the lower bounds, and we expect the lower bound of Theorem~\ref{thm:kite} to be the correct one. Proposition~\ref{prop:kite1} verifies this guess for curves of genus one.

In the current paper, we discuss two families of toric surfaces admitting reducible Severi varieties. The surfaces treated in Theorem~\ref{thm:general}, are quotients of toric surfaces by non-trivial subgroups of the torus that act freely on the complement of the zero-dimensional orbits. For such surfaces, the Severi varieties corresponding to {\em any} polarization and genus are reducible, as long as the genus is strictly positive and small enough with respect to the polarization. The surfaces treated in Theorem~\ref{thm:kite}, form an exceptional family. In particular, any surface in the second family that does not belong to the first family admits reducible Severi varieties for finitely many polarizations by \cite{L19}. We believe that the remaining examples, if any, also form finitely many exceptional families, and correspond to some narrow polygons.

In the appendix, we establish a connection between the Severi problem on kites and the topological classification of polynomials. For kites, the irreducible components of the Severi variety are in bijective correspondence with the components of the space of Laurent polynomials with given passports. We believe this connection to be fruitful for both topics.

\subsection{The idea of the proofs}
In order to prove the main results of this paper, we introduce two types of invariants that distinguish between the irreducible components of the Severi variety $V_{g,\Li}^\irr$. We then describe the admissible values of these invariants, and for each admissible value, construct an irreducible component realizing this value. We provide two proofs of the main theorems, one is topological, and another is tropical. We believe that both approaches will be useful in the ultimate classification of the irreducible components of Severi varieties on toric surfaces.

Our first invariant is a sublattice of the lattice of monomials of the toric surface, that can be described either in topological or in tropical terms. The second invariant makes sense only for polarized toric surfaces associated to kites. The set of nodes of an integral curve on such a surface admits a natural (unordered) partition, that varies continuously with the curve. Therefore, the absolute value of the difference between the cardinalities of the two blocks of nodes is an invariant of an irreducible component of $V_{g,\Li}^\irr$, that we call {\em the signature} of the component.

\medskip
	
\noindent {\bf Acknowledgements.} We are grateful to Georgios Dimitroglou Rizell, Henri Guenancia, and Grigory Mikhalkin for enlightening discussions. Special thanks are due to Grigory Mikhalkin for a discussion which lead us to the assertion of Theorem~\ref{thm:integrals}.

\section{Preliminaries, notation, and conventions}
\subsection{Toric geometry}
We assume that the reader is familiar with the basics of toric geometry, and refer to \cite{Dan78} for details. In particular, the construction of the polarized toric variety associated to a lattice polygon, the notion of the dual fan, the functoriality of toric varieties, and the structure of their orbit decomposition are assumed to be known. For a convex polygon $\Delta$, we denote by $\Delta^\circ$ and $\partial\Delta$ the interior and the boundary of $\Delta$, respectively.

Our default pair of dual lattices is $(\Z^2,\Z^2)$ with the standard pairing denoted by $\langle\cdot,\cdot\rangle$. Sometimes we consider different integral structures on $\R^2$ defined by various sublattices. In such cases, we follow the standard convention of toric geometry, and denote the lattice of monomials by $M$ and the lattice of cocharacters by $N$. For a toric surface $X$, we denote by $\Xb$ the open orbit, and by $\partial X:=X\setminus\Xb$ the boundary divisor. Then the canonical divisor of $X$ is given by $K_X=-\partial X$.

\subsection{Severi varieties}
For a projective polarized toric surface $(X,\Li)$, we denote by $V_{g,\Li}^\irr \subset \LL$ the Severi variety, i.e., the locus of integral nodal curves of geometric genus $g$ that contain no zero-dimensional orbits of $X$.

Severi varieties on toric surfaces are known to be smooth and equidimensional of codimension $\delta$ in $\LL$, where $\delta$ is the number of nodes, or equivalently, the difference between the arithmetic genus of $\Li$ and $g$; see, e.g.,  \cite[Theorem~1]{KST13}. If $\Li$ is associated to a lattice polygon $\Delta$ in $\R^2$, then its arithmetic genus is given by the number of inner lattice points in $\Delta$. Therefore, the dimension of any component of $V_{g,\Li}^\irr$ is given by
\begin{equation}\label{eq:dimsev}
|\partial\Delta\cap \Z^2|+g-1
\end{equation}
Furthermore, by \cite[Proposition~2]{KST13}, $V_{g,\Li}^\irr$ is open and dense in the locus of all integral curves of geometric genus $g$ in $\LL$ that contain no singular points of $X$.

We say that a curve in $C\subset X$ is \emph{torically transverse} if it intersects the boundary divisor $\partial X$ transversely, i.e., $C\cap X$ is reduced and contains no zero-dimensional orbits of $X$. It is well known that for a general $[C]\in V_{g,\Li}^\irr$, the curve $C$ is torically transverse; see e.g., \cite[Theorem~2.8]{Tyom}.

Recall that in genus zero, Severi varieties on toric surfaces admit natural parameterizations by irreducible rational varieties. Therefore, $V_{0,\Li}^\irr$ are irreducible for any polarized toric surface, see \cite[Proposition~4.1]{Tyom}.

\subsection{Conventions}

Throughout the paper we work over the field of complex numbers. The tropical arguments are applied after a base change to the field of Puiseux series.

For convenience, {\em we always assume that the lattice polygon, we work with, contains the origin as its vertex.} The latter can plainly be achieved by a translation, and hence does not restrict the generality. The advantage of this assumption is that all sublattices that appear in the statements and in the proofs become {\em linear} rather than affine.

\section{The case of genus one}
In this section we show that the bound of Theorem~\ref{thm:kite} is sharp in the case of genus one.

\begin{prop}\label{prop:kite1}
If $\Delta=\Delta_{k,k'}$ is a kite, and $(X,\Li)$ the corresponding polarized toric surface, then the number of irreducible components of $V_{1,\Li}^\irr$ is equal to $\#_{k,k',1}$.
\end{prop}

\begin{rem}\label{rem:kitearith}
Notice that a sublattice $M\subseteq\Z^2$ satisfying the condition (a) of Theorem~\ref{thm:general} for $\Delta=\Delta_{k,k'}$ is uniquely determined by its restriction to the $y$-axis, and contains the point $(0,2k)$. Furthermore, associating to a sublattice $M$ its index $r$, gives rise to a one-to-one correspondence between such sublattices and the positive common divisors of $k'+k$ and $2k$.
\end{rem}

\begin{proof}
By Remark~\ref{rem:kitearith}, the number $\#_{k,k',1}$ counts the positive common divisors $r$ of $k'+k$ and $2k$ for which $\delta_M(\Delta,1)\geqslant 0$, with multiplicities. Since $\delta_M(\Delta,1)=\big(\frac{k'+k}{r}-1\big)-1$, the latter inequality is equivalent to $r<k'+k$. We claim that all divisors $r$ are counted with multiplicity one. Indeed, if $r$ is even, then its multiplicity is $1$ according to the definition, and if $r$ is odd, then the multiplicity is equal to the number of integers $0\le\kappa\le \min\{\delta_M(\Delta,1), 1\}\le 1$ that are congruent to $\delta_M(\Delta,1)$ modulo two. Since such $\kappa$ is clearly unique, this proves the claim. Set now $d:=\gcd\{k'+k,2k\}=\gcd\{k'+k,k'-k\}$, and let $\sigma(d)$ be the number of positive divisors of $d$. Unless $k=k'$ or $k=0$, the inequality $r<k'+k$ is automatically satisfied, and the number $\#_{k,k',1}$ is therefore equal to $\sigma(d)$. Otherwise, the divisor $r=d$ has to be excluded and $\#_{k,k',1}$ is then equal to $\sigma(d)-1$.

Let $[C]\in V_{1,\Li}^\irr$ be a general point, and $E$ the normalization of the curve $C$. Denote by $O,P,Q,R$ the preimages in $E$ of the toric divisors corresponding to the sides of $\Delta$ with outer normals $(-k,-1),(k,-1),(k',1),(-k',1)$, respectively. Then $k'(Q-R)+k(P-O)$ and $Q+R-P-O$ are the divisors of the pull-backs of the coordinate functions on the torus, and hence, in the group law of the elliptic curve $(E,O)$ the following holds:
\begin{equation}\label{eq:elrel}
P=R+Q\; \text{ and }\; (k'+k)Q=(k'-k)R.
\end{equation}

Conversely, given an elliptic curve $(E,O)$ and three points $P,Q,R\in E$ satisfying \eqref{eq:elrel} such that $O,P,Q,R$ are distinct, pick rational functions $x,y$ such that ${\rm div}(x)=k'(Q-R)+k(P-O)$ and ${\rm div}(y)= Q+R-P-O$. Then the rational map $(x,y)\: E\setminus\{O,P,Q,R\}\to \ttor$ extends to a morphism $\phi\:E\to X$ and $[\phi_*E]\in |\Li|$. Furthermore, since the pull-back under $\phi$ of the boundary divisor $\partial X$ is reduced, $E$ is birational onto its image, and hence $[\phi(E)]\in V_{1,\Li}^\irr$. Notice that the functions $x$ and $y$ as above are unique up-to an action of $\C^*$. Therefore, the number of irreducible components of $V_{1,\Li}^\irr$ is equal to the number of irreducible components of the locus
$$V:=\{[E;O,Q,R]\in\mathcal M_{1,3}\,|\, (k'+k)Q=(k'-k)R\}$$ in the moduli space of genus-one curves with three marked points.

To describe the components of $V$, set $m:=\frac{k'+k}{d}>0, n:=\frac{k'-k}{d}\geqslant 0$, and let $a,b\in\Z$ be integers such that $am+bn=1$. For a divisor $d'|d$, let $\mathcal M_{1,2}[d']$ be the moduli space of elliptic curves with marked point and level-$d'$ structure, i.e., $\mathcal M_{1,2}[d']$ parametrizes elliptic curves with marked points $[E;O,S,T]$ where $O$ is the origin of $E$, $S$ a marked points, and $T$ a torsion point of order precisely $d'$. The projection $\mathcal M_{1,2}[d']\to \mathcal M_{1,1}[d']$ forgetting the point $S$ has irreducible fibers, and it is well known that $\mathcal M_{1,1}[d']$ is irreducible, since it is the quotient of the upper half-plane by the modular group $\Gamma_1(d')$, see, e.g., \cite[pp.439-440]{Sil09}. Hence $\mathcal M_{1,2}[d']$ is also irreducible.

Pick a general point $[E;O,S,T]\in \mathcal M_{1,2}[d']$.
Then, the point $O$ is distinct from $mS-bT$ since $m>0$ and $S$ is not a torsion point. For the same reason, the point $O$ is distinct from $nS+aT$ unless $n=0$ and $T=O$, that is $k'=k$ and $d'=1$. Eventually, the points $mS-bT$ and $nS+aT$ are distinct unless $(m-n)S$ is the torsion point $(a+b)T$. The latter occurs only if  $m=n$. However, $m$ and $n$ are relatively prime. Thus, $mS-bT=nS+aT$ if and only if $m=n=1$ (implying that $a+b=1$) and $T=O$, that is $k=0$ and $d'=1$. Besides these particular cases, the three points $O, nS+aT, mS-bT$ are distinct. Since the points are distinct, we can consider the rational map $\mathcal M_{1,2}[d']\to \mathcal M_{1,3}$ given by $[E;O,S,T]\mapsto [E;O,nS+aT,mS-bT]$. Its image $V_{d'}$ is irreducible, has the same dimension as $V$ and is contained in $V$ since
\[(k+k')(nS+aT)=dmnS+dmaT=dmnS=dmnS+dnbT=(k-k')(mS-bT).\]
Hence $V_{d'}$ is an irreducible component of $V$.

Next, let us show that $V=\bigcup_{d'|d}V_{d'}$, which implies that any irreducible component of $V$ is of the form $V_{d'}$ for some $d'$. Pick a general point $[E;O,Q,R]\in V$. By dimension reasons, $Q,R$ are not torsion points. Let us show that $[E;O,Q,R]\in\bigcup_{d'|d}V_{d'}$. Set $S':=aR+bQ$. Then $$d(Q-nS')=dQ-danR-dbnQ=dQ-damQ-dbnQ=dQ-dQ=O,$$ since, by definition, $dnR=(k'-k)R=(k'+k)Q=dmQ$ and $am+bn=1$. Thus, $Q-nS'$ is a $d$-torsion point. Similarly, $R-mS'$ is a $d$-torsion point. Set $T_1:=Q-nS'$ and $T_2:=R-mS'$. Since $am+bn=1$, the vectors $(n,m)$ and $(a,-b)$ form a free basis of $\Z^2$, and hence there exists a unique pair of $d$-torsion points $T,T'\in E$ such that $(T_1,T_2)=(nT',mT')+(aT,-bT)$. If we set $S:=S'+T'$, then $Q=nS+aT$ and $R=mS-bT$ as needed.

Finally, notice that if $d'|d$, $[E;O,S,T]\in \mathcal M_{1,2}[d']$, and $[E;O,Q,R]=[E;O,nS+aT,mS-bT]$ is its image in $V_{d'}$, then $d'$ is precisely the order of $T=mQ-nR\in E$, and hence the components corresponding to different $d'$-s are distinct. We have proved that the number of irreducible components of $V$, and hence also of $V_{1,\Li}^\irr$, is equal to $\sigma(d)$ unless $k'=k$ or $k=0$, in which case it is equal to $\sigma(d)-1$. Thus, in both cases, it coincides with  the number $\#_{k,k',1}$ as asserted.
\end{proof}

\section{The invariants}\label{sec:topinv}
In this section, we define the invariants that allow us to obtain the lower bounds in Theorems~\ref{thm:general} and \ref{thm:kite}. The first invariant is a certain sublattice of $\Z^2$, that can be defined either topologically or tropically for any component of the Severi variety $V_{g,\Li}^\irr$ on a toric surface. The second invariant is an integer, called \emph{the signature}, and is specific to the Severi varieties on the toric surfaces associated to kites.

\subsection{The sublattices associated to the irreducible components}\label{sec:toplat}
Throughout this section, $\Delta\subseteq \R^2$ is a convex lattice polygon one of whose vertices is the origin, $\Sigma$ its dual fan, and $(X,\Li)$ the associated polarized toric surface. Denote by $x^m$ the monomial functions of $X$ for $m\in\Z^2$, and by $\{n_i\}\subset\Z^2$ the set of primitive vectors along the rays of $\Sigma$.

Let $N\subseteq\Z^2$ be a sublattice containing all $n_i$'s. Since $\Sigma$ is a complete fan, $n_i$'s generate $\R^2$ as a vector space, and hence $N\subseteq \Z^2$ has finite index, which we denote by $r$. Moreover, since $N$ contains primitive vectors, and any primitive vector in $\Z^2$ can be completed to a free basis of $\Z^2$, it follows that the quotient $\Z^2/N$ is a cyclic group of order $r$. Set
\begin{equation}\label{eq:Mdef}
M:=\{m\in\Z^2\,|\,\forall n\in N, \; \langle n,m\rangle\in r\Z\},
\end{equation}
where $\langle\cdot,\cdot\rangle$ denotes the standard scalar product on $\R^2$. Then $M\subset\Z^2$ is the sublattice of index $r$ obtained from $N$ by a rotation by $\frac{\pi}{2}$. Furthermore, $M=r\cdot {\rm Hom}(N,\Z)$.

\begin{lem}\label{lem:affsublat}
In the above notation the following holds: $\partial\Delta\cap M=\partial\Delta\cap\Z^2$.
\end{lem}

\begin{proof}
Since the primitive integral vector along an edge of $\Delta$ is obtained from the primitive integral vector along the corresponding ray of $\Sigma$ by a rotation by $\frac{\pi}{2}$, it follows that $M$ contains the primitive integral vectors along all edges of $\Delta$, which implies the assertion.
\end{proof}

\subsubsection{The topological point of view}
For $[C]\in V_{g,\Li}^\irr$, set $\Cb:=C\cap\Xb$, and let $\tCb$ be the normalization of the curve $\Cb$. We denote by \emph{$N_C$} the image of the map $H_1(\tCb, \Z) \rightarrow H_1(\Xb, \Z)=\Z^2$ induced by the natural map $\tCb \rightarrow \Xb$. Notice that since the lattice of monomials of $X$ is $\Z^2$, the open orbit $\Xb$ is canonically trivialized, and therefore so is the first homology group $H_1(\Xb, \Z)$.

\begin{lem}\label{lem:topinv}
Let $V\subseteq V_{g,\Li}^\irr$ be an irreducible component. Then, for a general $[C]\in V$, the sublattice $N_C\subseteq \Z^2$ is independent of $[C]$, and contains all the $n_i$'s.
\end{lem}

\begin{proof}
Since $V$ parameterizes curves of the same genus, the tautological family $C_V\to V$ is equinormalizable, i.e., the normalization $\widetilde{C}_V\to V$ is a family of smooth curves of genus $g$, and hence topologically, it is a locally trivial fibration. Furthermore, since generically the family $C_V\to V$ is torically transverse, there exists an open dense subset $B\subseteq V$, over which the fibration $\widetilde{C}_B^\bullet\to B$ is locally trivial. Finally, since $B$ is connected, the image of the first homology group of the fibers of $\widetilde{C}_B^\bullet\to B$ in $H_1(\Xb, \Z)$ is constant. This proves the first part of the statement.

The second part follows from the fact that the fibers of $C_B\to B$ intersect all components of the boundary divisor, and the intersection is transverse. If $p$ is an intersection point of $C$ with the component corresponding to a ray $\rho$ of the dual fan, then the image of a small loop around $p$ corresponds to the class of the primitive integral vector along $\rho$ in $\Z^2=H_1(\Xb, \Z)$.
\end{proof}

Set $\emph{N(V)}:=N_C$ for a general $[C]\in V$. By the lemma, it is an invariant of the irreducible component $V$. Let \emph{$M(V)$} be the lattice dual to $N(V)$ in the sense of \eqref{eq:Mdef}, i.e., the lattice obtained from $N(V)$ by a rotation by $\frac{\pi}{2}$. Then $\partial\Delta\cap M(V)=\partial\Delta\cap\Z^2$ by Lemma~\ref{lem:affsublat}. We say that $(N(V),M(V))$ is the pair of sublattices \emph{topologically associated} to the irreducible component $V$.

\subsubsection{The tropical point of view}
Let $K$ be the field of Puiseux series over $\C$, and $\nu\:K\to \R\cup\{\infty\}$ its valuation. Let $[C]\in V_{g,\Li}^\irr(K)$ be a $K$-point. Denote by $\tC$ its normalization, and let $h\:\Gamma\to \R^2$ be the canonical tropicalization of the natural map $f\:\tC\to X$; cf. \cite{Tyo12} and \cite[\S~4.2]{CHT20}. Let us briefly remind how the canonical tropicalization works.

The underlying graph of $\Gamma$ is the dual graph of the stable reduction of the curve $\tC$ with marked points (or divisor) $f^*(\partial X)$: its vertices correspond to the components of the reduction, edges -- to the nodes, and legs -- to the marked points. The incidence relation among the vertices, edges, and legs is the natural one coming from the incidences between the components of the reduction, the nodes, and the specializations of the marked points. The length of the edge corresponding to a node $p$ of the reduction is defined to be $\nu(\lambda)$, where $\lambda\in K$ is such that, \'etale locally at $p$, the stable model of $\tC$ is given by $xy=\lambda$. Although $\lambda$ depends on the neighborhood, its valuations does not, and hence the length is well defined. The lengths of the legs are infinite.

The parametrization $h$ is the piecewise integral affine map uniquely determined by its slopes along the legs and its values at the vertices, which are defined as follows. If $l$ is the leg corresponding to a marked point $p$, then the slope $\frac{\partial h}{\partial\vec l}$ of $h$ along $l$ is given by the order of pole of the monomial functions at $p$, i.e., $\langle\frac{\partial h}{\partial\vec l},m\rangle=-\mathrm{ord}_p(f^*(x^m))$ for any monomial function $x^m$ on $X$. The notation $\vec l$ here indicates that the leg $l$ is oriented away from the vertex adjacent to it. For a vertex $v$, corresponding to an irreducible component $\tC_v$ of the reduction, $h(v)\in\mathbb{Q}^2$ is defined to be the vector for which $\langle h(v),m\rangle=\nu(\lambda_m)$ for any $m\in\Z^2$, where $\lambda_m\in K$ is a scalar such that $\lambda_mf^*(x^m)$ restricts to a non-zero rational function on $\tC_v$. Once again, although $\lambda_m$ is not uniquely defined, its valuation is, and hence the definition makes sense.

For a parametrized tropical curve $h\:\Gamma\to \R^2$, if $\vec e$ is an oriented edge or leg of $\Gamma$, then the slope $\frac{\partial h}{\partial\vec e}$ is an integral vector, whose integral length is the stretching factor of the piecewise integral affine map $h$. We denote by $N(\Gamma)\subseteq\Z^2$  the sublattice generated by all the slopes $\frac{\partial h}{\partial\vec e}$.  If the degree of $\Gamma$ is dual to $\Delta$, and $\Delta=\cup\Delta_i$ is the Legendre dual subdivision of $\Delta$ associated to $\Gamma$, then the affine sublattice $M(\Gamma)\subseteq \Z^2$ generated by all vertices of the subdivision is dual to $N(\Gamma)$ in the sense of \eqref{eq:Mdef}, i.e., it is obtained from $N(\Gamma)$ by a rotation by $\frac{\pi}{2}$. If $h\:\Gamma\to \R^2$ is the tropicalization of $f\:\tC\to X$ corresponding to a point $[C]\in V_{g,\Li}^\irr(K)$, then we set $\emph{N_C^\trop}:=N(\Gamma)$ and $\emph{M_C^\trop}:=M(\Gamma)$.

For an irreducible component $V\subseteq V_{g,\Li}^\irr$, let $B\subseteq V$ be the open dense locus of torically transverse curves. Set $N^\trop(V):=\sum_{[C]\in B(K)}N_C^\trop\subseteq \Z^2$ and $M^\trop(V):=\sum_{[C]\in B(K)}M_C^\trop\subseteq \Z^2$. We call this pair of sublattices the sublattices \emph{tropically associated} to the irreducible component $V$. Notice that since for any $[C]\in B$, the curve $C$ is torically transverse, it follows that the slopes of all legs of its tropicalization $h\:\Gamma\to\R^2$ are primitive integral vectors belonging to the rays of the dual fan $\Sigma$. Therefore, the sublattice $N_C^\trop$ contains all the $n_i$'s, and $\partial\Delta\cap M^\trop(V)=\partial\Delta\cap \Z^2$.

\begin{rem}
One can show that the topological and the tropical lattice invariants coincide, i.e., $(N(V),M(V))=(N^\trop(V),M^\trop(V))$. Since we will not use this, we only indicate the main idea and leave the details to an interested reader. The general idea is to compare the lattices using the curves in $V$ defined over the  subfield $F\subset K$ of convergent Puiseux series, and to prove that (i) all possible tropicalizations of curves in $V$ are realized by curves defined over $F$, and (ii) the irreducible component $V$ contains a Mumford curve $C$ defined over $F$, whose tropicalization is regular. Property (i) allows one to show that $N^\trop(V)\subseteq N(V)$, and property (ii) -- to deduce the equality.
\end{rem}

\subsection{The nodal partition and the signature}\label{sec:nodalpart}
Let $\Delta=\Delta_{k,k'}$ be a kite, and $0\le g\le |\Delta^\circ\cap\Z^2|$ an integer. Set $\delta:=\delta_{\Z^2}(\Delta,g)$. Then $\delta$ is the number of nodes of the curve $C$ for any $[C]\in V_{g,\Li}^\irr$. We claim that the set of $\delta$ nodes of $C$ admits a natural unordered partition into two blocks, which we call \emph{the nodal partition} of $C$. Indeed, let $f \in H^0(X,\Li)$ be a Laurent polynomial defining $C$. Since $C$ is torically transverse, all nodes of $C$ belong to $\Xb$, and hence are given by the system of equations
$f=z\partial_z(f)=w\partial_w(f)=0.$ Furthermore, since $\Delta$ is a kite, the polynomial $f$ is of the form $f(z,w)=az^{-1}w^k+p(w)+bzw^k$, where $p(w)$ is a polynomial of degree $k+k'$, and $a,b\in \C^*$. Thus, any node $(z,w)\in C$ satisfies: $-az^{-1}+bz=0,$ and hence $z^2=\frac{a}{b}$. We conclude, that the set of nodes of $C$ admits a natural partition into two blocks, and a pair of nodes belongs to the same block if and only if their $z$-coordinates coincide.

Plainly, the nodal partition varies continuously in the tautological family $C_{V_{g,\Li}^\irr}\to V_{g,\Li}^\irr$. Furthermore, the number of nodes in the block corresponding to a given value of $z$ is equal to the number of critical points of $\frac{p(w)}{w^k}$ with critical value $\frac{p(w)}{w^k}=-2bz$. Therefore, the number of nodes in each block is bounded from above by $\lfloor\frac{k+k'}{2}\rfloor=\lfloor\frac{\delta+g+1}{2}\rfloor=\lceil\frac{\delta+g}{2}\rceil$.

Notice that the monodromy acts naturally on the two blocks of nodes of $C$, and the blocks get interchanged, when $\frac{a}{b}$ travels along a loop around $0\in \C$. Therefore, the partition is unordered, and the induced integer partition $\delta=\delta_1+\delta_2$, where $\delta_1\geqslant\delta_2$, is an invariant of the irreducible component $V\subseteq V_{g,\Li}^\irr$ containing $[C]$. The \emph{signature} of $C$ and of $V$, is defined to be $\emph{\kappa(V)}:=\kappa(C):=\delta_1-\delta_2$. It is clear from the definition, that for any irreducible component $V\subseteq V_{g,\Li}^\irr$, its signature satisfies the following properties: $\kappa(V)$ is congruent to $\delta$ modulo $2$, and $\kappa(V)+\delta=2\delta_1\le 2 \lceil\frac{\delta+g}{2}\rceil$. Thus, $$0\le \kappa(V)\le \min\left\{\delta,2\left\lceil\frac{\delta+g}{2}\right\rceil-\delta\right\}.$$
In particular, if $k+k'$ is even, then $\delta+g$ is odd, and hence $0\le\kappa(V)\le\min\{\delta,g+1\}$. And if $k+k'$ is odd, then $\delta+g$ is even, and hence $0\le\kappa(V)\le\min\{\delta,g\}$.

\section{Tropical proofs of the main results}

Let $\Delta\subset \R^2$ be a lattice polygon containing the origin as one of its vertices, and $\Sigma$ its dual fan. Let $N\subseteq\Z^2$ be a sublattice containing the primitive vectors along all the rays of $\Sigma$. Recall that $N$ is a sublattice of a finite index $r$, and the quotient $\Z^2/N$ is isomorphic to the cyclic group $\Z/r\Z$. Let $M\subseteq \Z^2$ be the sublattice obtained from $N$ by a rotation by $\frac{\pi}{2}$. Then $M=r\cdot {\rm Hom}(N,\Z)$ and $N=r\cdot {\rm Hom}(M,\Z)$. Since the origin is a vertex of $\Delta$, the sublattice $M$ contains all vertices of $\Delta$. Denote by $\mu_r$ the kernel of the natural surjective homomorphism $N\otimes_\Z\C^*\to \Z^2\otimes_\Z\C^*=(\C^*)^2$.

\begin{lem}\label{lem:gentoricstuff}
Let $(X,\Li)$ and $(X',\Li')$ be the polarized toric varieties associated to the polygon $\Delta$ with respect to the lattices $\Z^2$ and $M$. Then,

\begin{enumerate}
\item There is a natural $\mu_r$-equivariant projection $\pi\colon X'\to X$, and $X=X'/\mu_r$;

\item The action of $\mu_r$ on the one-dimensional orbits of $X'$ is free;

\item $\pi^*\Li\cong(\Li')^{\otimes r}$ and $\pi_*|\Li'|=|\Li|$.
\end{enumerate}
\end{lem}
The assertions of the lemma are well-known in toric geometry. We only include a sketch of its proof for the convenience of the reader.
\begin{proof}
Since $N=r\cdot {\rm Hom}(M,\Z)$, the variety $X'$ is canonically isomorphic to the toric variety associated to the fan $\Sigma$ with respect to the integral structure given by $N$. By the functoriality of toric varieties, we obtain a natural morphism $\pi\: X'\to X$ compatible with the actions of $N\otimes \C^*$ and $(\C^*)^2$. Furthermore, $\pi\:X'\to X$ is a Galois cover with Galois group $\mu_r$ by, e.g., \cite[\S~2.6.2]{Dan78}.

To verify (2), let $\rho$ be a ray in $\Sigma$, and $O\subset X$, $O'\subset X'$ the corresponding one-dimensional orbits. Then $O={\rm Spec}\left(\C[\Z^2\cap\rho^\perp]\right)$ and $O'={\rm Spec}\left(\C[\frac{1}{r}M\cap\rho^\perp]\right)$. Since $N$ contains the primitive integral vector of $\rho$, it follows that $\Z^2\cap\rho^\perp$ has index $r$ in $\frac{1}{r}M\cap\rho^\perp$, and hence the degree of $\pi_{|_{O'}}\:O'\to O$ is $r$, i.e., $\mu_r$ acts freely on $O'$.

If $[C]\in |\Li|$, then the intersection of $C$ with the divisor $D_\rho\subset X$ corresponding to a ray $\rho$ is given by the integral length of the dual side of $\Delta$. Therefore, $\pi^*(C).D'_\rho$ is $r$ times bigger, and hence coincides with the integral length of $\Delta$ with respect to $\frac{1}{r}M$, or equivalently, the integral length of $r\Delta$ with respect to $M$. We conclude that $\pi^*\Li$ and $(\Li')^{\otimes r}$ belong to the same class in the Neron-Severi group ${\rm NS}(X')$. But ${\rm NS}(X')\cong{\rm Pic}(X')$, and therefore $\pi^*\Li\cong(\Li')^{\otimes r}$. The second part of assertion (3) is now clear.
\end{proof}

\begin{lem}\label{lem:propofSev}
Let $(X,\Li)$ and $(X',\Li')$ be as in Lemma~\ref{lem:gentoricstuff}. Denote by $V^\irr_{g,\Li'}$ the Severi variety of integral genus $g$ curves in the linear system $|\Li'|$ on $X'$, and let $V'\subseteq V^\irr_{g,\Li'}$ be an irreducible component. Denote by $V$ to the locus of reduced curves in $\pi_*(V')$. Then,
\begin{enumerate}
\item $\dim(V^\irr_{g,\Li'})=\dim(V^\irr_{g,\Li})$;

\item $V\subseteq V^\irr_{g,\Li}$ is an irreducible component;

\item If $\Delta=\Delta_{k,k'}$ is a kite, then $\kappa(V')=\kappa(V)$ if $r$ is odd, and $\kappa(V)=g+1$ if $r$ is even;

\item $(N^\trop(V),M^\trop(V))=(N^\trop(V'),M^\trop(V'))$.
\end{enumerate}
\end{lem}

\begin{proof}
By \cite[Theorem~1]{KST13}, the Severi varieties $V^\irr_{g,\Li}$ and $V^\irr_{g,\Li'}$ are equidimensional, and
$$\dim(V^\irr_{g,\Li})=|\partial\Delta\cap \Z^2|+g-1=|\partial\Delta\cap M|+g-1=\dim(V^\irr_{g,\Li'}),$$
since $|\partial\Delta\cap \Z^2|+g-1=|\Li|-\delta_{\Z^2}(\Delta,g)$ and $|\partial\Delta\cap M|+g-1=|\Li'|-\delta_{M}(\Delta,g)$. This proves (1).

To prove (2), notice that since $\pi_*\:|\Li'|\to|\Li|$ is a finite morphism, we have the equality of dimensions $\dim(\pi_*(V'))=\dim(V^\irr_{g,\Li})$. Thus, $V$ is dense in an irreducible component of $V^\irr_{g,\Li}$. On the other hand, if $[C]$ belongs to this component, then $\pi^{-1}(C)$ is a reduced curve in the linear system $|(\Li')^{\otimes r}|$. Furthermore, $\pi^{-1}(C)$ is a specialization of a $\mu_r$-orbit of an element of $|\Li'|$, and hence it is a $\mu_r$-orbit of some $[C']\in |\Li'|$. Plainly, $\pi\:C'\to C$ is a birational map, and hence $[C']\in V'$. Assertion (2) now follows.

The proof of (3) is a rather long but straight-forward computation. We start with the case when $M$ contains the point $(0,k)$, i.e., $\Delta$ is a kite also with respect to the sublattice $M$. In this case, the action of $\mu_r$ on a point $(z,w)\in \Xb$ is given by $\xi(z,w)=(\xi z,w)$, and the projection $\pi$ is given by $\pi(z,w)=(z^r,w)$.

Let $[C']\in V'$ be general, and set $C:=\pi(C')$. Then $[C]\in V$, and hence the curve $C$ is nodal, torically transverse, and has geometric genus $g$. Since $\pi^{-1}(C)=\cup_{\xi\in\mu_r}\xi(C')$, each node of $C$ has $r$ preimages, which are either the $\mu_r$-orbit of a node of $C'$, or the $\mu_r$-orbit of a point of intersection of $C'$ with $\xi(C')$ for some $1\ne \xi\in\mu_r$. Recall that two nodes belong to the same block of the nodal partition if and only if they have the same $z$-coordinates, and they belong to different blocks if and only if their $z$-coordinates differ by a sign. Therefore, if $r$ is odd, then a pair of nodes of $C'$ belongs to the same block of the nodal partition if and only if their images do so; and if $r$ is even, then all nodes of $C'$ are mapped to the same block of the nodal partition of $C$.

Let $p\in C$ be a node, whose preimage is a pair of points $\{p'_1,p'_2\}\subset C'$. Then there exists a unique $\xi\in\mu_r$ such that $p'_1\in C'\cap\xi(C')$, and hence $p'_2\in C'\cap \xi^{-1}(C')$. Let $cz^{-1}w^{k/r}+q(w)+ezw^{k/r}$ be a section of $\Li'$ defining $C'$. Then the intersection $C'\cap\xi(C')$ is given by the system of equations:
$$cz^{-1}w^{k/r}+q(w)+ezw^{k/r}=c\xi^{-1}z^{-1}w^{k/r}+q(w)+e\xi zw^{k/r}=0,$$ or equivalently $cz^{-1}w^{k/r}+q(w)+ezw^{k/r}=cz^{-1}-e\xi z=0$. Since $C=\pi(C')$ is nodal, the intersection $C'\cap\xi(C')$ is transverse. Thus, for each solution of $cz^{-1}=e\xi z$, we have $\frac{k+k'}{r}$ solutions of $cz^{-1}w^{k/r}+q(w)+ezw^{k/r}=0$.

Assume that $r$ is odd. To prove that $\kappa(V')=\kappa(V)$, it remains to show that each block of the nodal partition of $C$ contains the same amount of nodes, whose preimages are pairs of points in $C'$. Since $r$ is odd, $\xi\ne\xi^{-1}$ for any $1\ne\xi\in\mu_r$. Therefore, it suffices to show that for any $\xi\ne 1$, the points of intersection $C'\cap\xi(C')$ contribute equally to the two blocks of the nodal partition of $C$, but the latter is clear because the $r$-th powers of the two solutions of $cz^{-1}=e\xi z$ differ by a sign, and therefore each one of them contributes $\frac{k+k'}{r}$ nodes to the corresponding block.

Assume now that $r$ is even. In this case, all nodes of $C'$ are mapped to the block of nodes of $C$ for which $z=\left(\frac{c}{e}\right)^{r/2}$. Similarly, for any $\xi\ne 1$, all intersection points of $C'\cap\xi(C')$ are mapped to the block in which $z=\left(\frac{c}{e}\xi\right)^{r/2}$. Thus,
$$\kappa(V)=\left|\frac{k+k'}{r}\sum_{1\ne\xi\in\mu_r}\xi^{r/2}+\delta_M(\Delta,g)\right|=\left|\frac{k+k'}{r}\sum_{\xi\in\mu_r}\xi^{r/2}-1-g\right|=g+1,$$
since $|C'\cap \xi(C')|=2\frac{k+k'}{r}$ for all $\xi\ne 1$.

It remains to prove assertion (3) in the case when $M$ does not contain the point $(0,k)$. By Remark~\ref{rem:kitearith}, $r$ is a common divisor of $k+k'$ and $k-k'$. In particular, $r$ is a divisor of $2k$ but it does not divide $k$ since otherwise $(0,k)\in M$. Therefore, $r$ is necessarily even. For the computation, it is more convenient to apply an affine automorphism of $\Z^2$ so that $\Delta$ becomes the polygon with vertices $(0,0),(-1,0),(1,2k),(0,k+k')$. Then $\mu_r$ again acts by $\xi(z,w)=(\xi z,w)$ and $\pi(z,w)=(z^r,w)$.

Let $az^{-1}+p(w)+bzw^{2k}$ be a section of $\Li$ defining the curve $C$. Then the nodes of $C$ satisfy $-az^{-2}+bw^{2k}=0$, and hence two nodes of $C$ belong to the same block if and only if they have the same value of $zw^{k}$. Similarly, if $cz^{-1}+q(w)+ezw^{2k/r}$ is a section of $\Li'$ defining the curve $C'\subset X'$, then the nodes of $C'$ satisfy $z^2w^{2k/r}=\frac{c}{e}$. Thus, the image $(z^r,w)$ of every node of $C'$ satisfies $z^rw^{k}=\left(\frac{c}{e}\right)^{r/2}$, and hence belongs to the same block of the nodal partition of $C$. Other nodes of $C$ correspond to the points of intersection of $C'$ with $\xi(C')$ for various $1\ne \xi\in\mu_r$, and the computation, identical to the one we did above, shows that $\kappa(V)=g+1$.

Finally, let us prove (4). Let $K$ be the field of Puiseux series, and $[C']\in V_{g,\Li'}^\irr(K)$ be such that $\pi_*[C']\in V_{g,\Li}^\irr(K)$. Set $C:=\pi(C')$. Then $C'\to C$ is birational, and hence the normalizations $\widetilde{C}'$ and $\widetilde{C}$ are canonically isomorphic. Let us show that the tropicalizations of $f'\:\widetilde{C}'\to X'$ and $f\:\widetilde{C}\to X$ coincide. Under the identification $\widetilde{C}'=\widetilde{C}$, we have $f=\pi\circ f'$. Thus, the abstract tropical curves $\Gamma'$ and $\Gamma$ coincide, and for any $m\in\Z^2$, we have $f^*(x^m)=(f')^*(x^m)$. Therefore, the parametrizations $h'$ and $h$ coincide too by their very definition.  We conclude that $N^\trop_C=N^\trop_{C'}$.

We have seen in the proof of assertion (2) above, that for any $[C]\in V(K)$ there exists $[C']\in V'(K)$ such that $\pi(C')=C$. Thus, $N^\trop_C=N^\trop_{C'}$, and hence $N^\trop(V')\subseteq N^\trop(V)$. Vice versa, for $[C']\in V'(K)$, consider a general one-parameter deformation $[C'_t]$ in $V'(K)$. Then, for any $t$ in a small punctured neighborhood of $0$, the pushforward $\pi_*([C'_t])$ belongs to $V(K)$, and hence $N^\trop(V)\subseteq N^\trop_{C'_t}$. On the other hand, in any punctured neighborhood of $0$, there exist $t$ such that the tropicalization of $C'_t$ coincides with that of $C'$; see, e.g., \cite[Theorem~4.6]{CHT20}. Thus, $N^\trop(V)\subseteq N^\trop_{C'}$, and hence $N^\trop(V)\subseteq N^\trop(V')$. It follows now that $N^\trop(V')=N^\trop(V)$, and hence also $M^\trop(V')=M^\trop(V)$, as asserted.
\end{proof}

\subsection{Proof of Theorem~\ref{thm:general}}
Without loss of generality we may assume that $\Delta$ contains the origin as one of its vertices, and the sublattices we are interested in are linear. Let $M\subseteq\Z^2$ be a sublattice satisfying the conditions (a) and (b) of the theorem. It is sufficient to prove that for any $1\le g'\le |\Delta^\circ\cap M|$, there exists a convex $M$-integral triangulation $\Delta=\cup\Delta_i$, whose set of vertices  contains $\partial\Delta\cap M$ and generates $M$, and the number of vertices of the triangulation in $\Delta^\circ$ is $g'$. Indeed, given such a triangulation for $g'=g$, the dual tropical curve $h\:\Gamma\to\R^2$ is a trivalent irreducible curve of genus $g$, and hence regular by \cite[Proposition~2.23]{Mikh05}. Let $(X',\Li')$ be the polarized toric surface corresponding to the polygon $\Delta$ with respect to the lattice $M$. Then $h\:\Gamma\to\R^2$ is realizable by an irreducible nodal algebraic curve $f\:C\to X'$ of genus $g$ over the field of Puiseux series; see, e.g., \cite[Lemma~3.12]{Shu05} or \cite[Theorem~6.2]{Tyo12}. Let $V'$ be the component of the Severi variety $V_{g,\Li'}^\irr$ containing $[C]$, and $V:=\pi_*(V')$ the corresponding component of $V_{g,\Li}^\irr$. Then $M^\trop(V)=M^\trop(V')=M$ by Lemma~\ref{lem:propofSev} (4), which implies the asserted bound on the number of irreducible components of $V_{g,\Li}^\irr$.

Let us construct the desired $M$-triangulations of $\Delta$. Since $\partial\Delta\cap M=\partial\Delta\cap \Z^2$, there exists $m\in \Delta^\circ\cap M$ such that $\{m\}\cup (\partial\Delta\cap M)$ generates $M$. Thus, the triangulation with vertices $\{m\}\cup (\partial\Delta\cap M)$ is the desired convex triangulation in the case $g'=1$. We proceed by induction. Once a desired triangulation $\Delta=\cup\Delta_i$ is constructed for $1\le g'< |\Delta^\circ\cap M|$, we construct the triangulation for $g'+1$ in the following way: if there is a triangle $\Delta_i$ containing a lattice point in its interior, then we add this point as a new vertex of the triangulation. Otherwise there is a pair of triangles containing a lattice point in the interior of their common edge, and we add this point as a new vertex of the triangulation. Plainly, the new triangulation is convex in both cases, and the number of vertices of the new triangulation in the interior of $\Delta$ is $g'+1$.\qed

\subsection{Proof of Theorem~\ref{thm:kite}}\label{sec:proofkite}
Let $M\subseteq\Z^2$ be a sublattice of index $r$. A pair $(M,\kappa)$ is called \emph{admissible} if and only if $M$ satisfies conditions (a) and (b) of Theorem~\ref{thm:general}, and one of the following conditions holds
\begin{itemize}
\item the index $r$ is even and $\kappa=g+1$, or
\item the index $r$ is odd, $0\le\kappa\le \min\{\delta_M(\Delta,g), g\}$, and $\kappa$ is congruent to $\delta_M(\Delta,g)$ modulo $2$.
\end{itemize}
To prove the theorem, it is sufficient to construct for any admissible pair $(M,\kappa)$, an irreducible component $V\subseteq V_{g,\Li}^\irr$ such that $(M^\trop(V),\kappa(V))=(M,\kappa)$. Pick an admissible pair $(M,\kappa)$, and let $(X',\Li')$ be the polarized toric surface associated to the polygon $\Delta$ with respect to the lattice $M$.

Assume first, that $r$ is even. Then $\kappa=g+1$. As in the proof of Theorem~\ref{thm:general}, there exists an irreducible component $V'\subseteq V_{g,\Li'}^\irr$ such that $M^\trop(V')=M$. Set $V:=\pi_*(V')$. Then $M^\trop(V)=M$ and $\kappa(V)=g+1=\kappa$ by Lemma~\ref{lem:propofSev} (3)-(4).

Assume now that $r$ is odd. Then $\kappa\le \min\{\delta_M(\Delta,g), g\}$, and hence $$\kappa+g\le \delta_M(\Delta,g)+g=|\Delta^\circ\cap M|=\frac{k+k'}{r}-1.$$ Consider the $M$-integral triangulation of $\Delta=\Delta_{k,k'}$, whose inner vertices are the points
$$k'-2r,k'-4r,\dotsc,k'-2\kappa r,k'-2\kappa r-r,k'-2\kappa r-2r,\dotsc, k'-2\kappa r-(g-\kappa)r$$
on the $y$-axis. Notice that $k'-2\kappa r-(g-\kappa)r=k'-(\kappa+g)r\geqslant k'-(k+k'-r)=r-k$, and therefore such triangulation exists. The triangulation subdivides the interval $[-k,k']$ on the $y$-axis into $g+1$ subintervals. The top $\kappa$ of them have $M$-integral length $2$, the next $g-\kappa$ intervals have $M$-integral length one, and the bottom interval has $M$-integral length
$$\frac{k'-2\kappa r-(g-\kappa)r+k}{r}=\delta_M(\Delta,g)+g+1-(\kappa+g)=\delta_M(\Delta,g)-\kappa+1,$$ which is odd since $\kappa$ is congruent to $\delta_M(\Delta,g)$ modulo $2$. In particular, the vertices of the subdivision generate the affine lattice $M$.

Let $h\:\Gamma\to\R^2$ be the trivalent tropical curve dual to the triangulation constructed above. We will prove that it is liftable to an algebraic curve $[C']\in V_{g,\Li'}^\irr$ for which $\kappa(C')=\kappa$ by using Shustin's refined tropicalization \cite[\S~3.5]{Shu05}. Let us start by picking an arbitrary tropical datum as in \cite[\S~3.7]{Shu05} with nodal amoeba $h(\Gamma)$. Then, we can apply \cite[Theorem~5]{Shu05} to lift it to an algebraic curve $C'$ over the field $K$. Furthermore, since the triangles in the dual subdivision have no inner integral points, the  nodes of $C'$ are in natural one-to-one correspondence with the nodes of the $z$-refinements of the tropicalization of $C'$. This allows us to control the signature $\kappa(C')$ in the following way.

For each of the top $\kappa$ intervals on the $y$-axis, its $M$-integral length is $2$, and hence there are two possible $z$-refinements as in \cite[Lemma~3.9]{Shu05}, and each of them contributes a single node to $C'$. The next $g-\kappa$ intervals have $M$-integral length one, and hence contribute no nodes at all. Finally, the bottom interval has an odd integral length, and hence contributes an even number of nodes. Furthermore, the nodes of the $z$-refinement corresponding to the bottom interval contribute equally to the two blocks of nodes of  $C'$. Indeed, if $h$ is the Chebyshev polynomial of an odd degree (or any of its twists as in the proof of \cite[Lemma~3.9]{Shu05}), then it has the same number of critical points with critical values $1$ and $-1$, and $\pm 1$ are the only critical values of $h$. Therefore, half of the nodes of the bottom $z$-refinement, which is given by the polynomial $az^2+bzh(w)+c$, have one value of the $z$-coordinate, and half -- the opposite value.

Notice however, that changing the choice of the $z$-refinement corresponding to one of the top $\kappa$ intervals, changes the block to which the refining curve contributes its node. Thus, after replacing some of the first $\kappa$ refining curves with their twists, we can make sure that {\em all} nodes of the first $\kappa$ refining curves contribute to {\em the same} block of nodes of $C'$. Therefore, by \cite[Lemma~3.12]{Shu05}, there exists  $[C']\in V_{g,\Li'}^\irr$ such that $\kappa(C')=\kappa$. Furthermore, since the vertices of the triangulation generate $M$, we also have: $M^\trop(C')=M$.

To finish the proof, let $V'\subseteq V_{g,\Li'}^\irr$ be the irreducible component containing $[C']$, and set, as usual, $V:=\pi_*(V')$. Then $V\subseteq V_{g,\Li}^\irr$ is an irreducible component, and since $r$ is odd, we have
$$(M^\trop(V),\kappa(V))=(M^\trop(V'),\kappa(V'))=(M,\kappa)$$
by Lemma~\ref{lem:propofSev} (3)-(4). This completes the proof. \qed

\section{Deformation of curves in toric surfaces}\label{sec:def}

The main goal of this section is to introduce local coordinates on $V_{g,\Li}^\irr$ that we can use to construct deformations to curves of lower genus. In particular, we will prove Theorem \ref{thm:integrals}.

For the convenience of the reader, we recollect the relevant material introduced in Section \ref{sec:mainres}. Fix $[C] \in V_{g,\Li}^\irr$, and pick a collection $\Upsilon:=\{\gamma_1,\dotsc,\gamma_g\}$ of  oriented simple closed curves in the smooth locus of $\Cb$ such that each $\gamma_j$ is contractible in $\Xb$ and such that $\Upsilon$ generates a $g$-dimensional subspace of $H_1(C,\Z)$. Since $\Xb$ is acyclic, each $\gamma_j$ bounds a smooth disc $M_j\subset \Xb$ that is unique up to homotopy. Denoting $\Omega:= \frac{dz\wedge dw}{zw}$ the holomorphic $2$-form on $\Xb\simeq \ttor$, we can define the integrals
\begin{equation}\label{eq:integrals1}
\int_{M_j} \Omega \; , \quad j=1,\dotsc,g.
\end{equation}
Since the $2$-form $\Omega$ is closed, the above integrals do not depend on the choice of the smooth membranes $M_j$'s. Additionally, the integrals do not depend on the choice of $\gamma_j$'s within a fixed isotopy class, as the form $\Omega$ is identically zero on $C$.

Denote by $W\subset V_{g,\Li}^\irr$ the locus of curves having given tangency profile with the boundary divisor. For any  $[C] \in W$ and any open neighborhood $U\subset W$ of $[C]$ such that the monodromy action of $\pi_1(U,C)$ on $C$ acts as the identity on $\Upsilon$,
we can carry the collection $\Upsilon$ to a collection $\Upsilon'$ of simple closed curves in $[C']\in U$. Therefore, the integrals \eqref{eq:integrals1} define a map
\[\Phi_\Upsilon\: U \rightarrow \C^g.\]
At last, denote by $U^{C}\subset U$ the subset of curves $[C']$ such that $C'\cap \partial X= C\cap \partial X$. The following statement is a slight generalization of Theorem \ref{thm:integrals}.

\begin{thm}\label{thm:integrals1}
For $[C]\in U$ and $\Upsilon$ as above, the restriction of the map $\Phi_\Upsilon$ to $U^{C}$ is a local diffeomorphism. Moreover, the integrals \eqref{eq:integrals1} together with the coordinates of the intersection points of $C$ with the boundary divisor of $X$ provide a local system of coordinates on $W$.
\end{thm}

In order to prove Theorem \ref{thm:integrals1}, we will need a description of the space of holomorphic $1$-forms on the normalization of $C$.  We denote by $\pi\: \widetilde{C} \rightarrow C$ the normalization map. Recall that the \emph{adjoint polygon} $\Delta_a$ is defined by $\Delta_a :=\conv(\Delta^\circ)$.

\begin{lem}\label{lem:1form}
For any curve $[C]\in\sv{g,\Li}^{\irr}$ whose set of nodes is contained in $\Xb$, the space of holomorphic $1$-forms on the normalization $\widetilde{C}$ is isomorphic to the linear subspace $L\subset H^0(X,\Li)$ consisting of the Laurent polynomials $h(z,w)$ vanishing at the nodes of $C$ and having the Newton polygon contained in $\Delta_a$. The isomorphism is given by the map
\[h(z,w) \mapsto \xi_h:= \pi^*\Big(\dfrac{h(z,w)}{\partial_wf(z,w)\cdot zw}dz\Big)\]
where $f\in H^0(X,\Li)$ is a Laurent polynomial defining $C^\bullet$.
\end{lem}

\begin{proof}
Observe that $L$ is isomorphic to the tangent space to $U^{C}$ at the point $[C]$, which is smooth of dimension $g$. It follows that $L$ has the expected dimension. Moreover, the two meromorphic $1$-forms $\xi_{h}$ and $\xi_{\tilde h}$ are linearly independent provided that $h$ and $\tilde h$ are. It remains to show that $\xi_h$ is indeed holomorphic.

Since $h\in L$ vanishes at the nodes of $C$, the form $\xi_h$ is holomorphic in a neighborhood of the branches of the nodes. This can be seen using a local parametrization of the branches in $\Xb$. The fact that $\xi_h$ is holomorphic on the rest of the curve is proven in \cite[Lemma $4.3$]{CL1} (the proof is identical except that the charts $\C^2$ are to be replaced with charts $\C^*\times \C$).
\end{proof}

\begin{proof}[Proof of Theorem \ref{thm:integrals1}]
Assume towards the contradiction that the restriction of $\Phi_\Upsilon$ to $U^{C}$ is not a local diffeomorphism at $[C]$. Then, there exists a non-zero polynomial $h$ as in Lemma \ref{lem:1form} such that for
$$C_t:=\overline{\{(z,w)\in \Xb \, \vert \, f(z,w)=t\cdot h(z,w)\}},$$
the derivative of the smooth function $t\mapsto \Phi_{\Upsilon}([C_t])$ vanishes at $0$. For any $j=1,\dotsc,g$, fix an arbitrarily small neighborhood $V_j\subset \Xb$ of $M_j$ and an arbitrarily small $\varepsilon>0$. In order to compute the $j^{th}$ coordinate of the derivative of $\Phi_{\Upsilon}([C_t])$, we introduce a $C^\infty$ complex-valued, time-dependent vector field
\[\chi(z,w,t)=\alpha(z,w,t) \partial z + \beta(z,w,t) \partial w\]
defined on $V_j$, in time $\vert t\vert <\varepsilon$, whose associated flow $\phi_t\:V_j\rightarrow \Xb$  maps
$C_0\cap V_j$ to $C_t$ and such that $\phi_t(M_j)$ is a smooth membrane. In particular, the $j^{th}$ coordinate of $\Phi_{\Upsilon}([C_t])$ is given by the integral $\int_{\phi_t(M_j)}\Omega$. Any such vector field, if it exists, satisfies the relation
\begin{equation}\label{eq:relationchi}
\alpha\cdot \partial_z f + \beta \cdot \partial_w f = h+ t (\alpha\cdot \partial_z h + \beta \cdot \partial_w h),
\end{equation}
since any integral curve $t\mapsto (z(t),w(t))$ of $\chi$ satisfies the relation $f(z(t),w(t))=t\cdot h(z(t),w(t))$ whose derivative with respect to $t$ is \eqref{eq:relationchi}.
By the assumption made on $\Phi_\Upsilon$, we have
\begin{equation}\label{eq:Liederivative}
0 \; = \; \frac{d}{dt} \left. \int_{\phi_t(M_j)}\Omega \; \right|_{t=0}\; = \;  \frac{d}{dt} \left. \int_{M_j}\phi_t^\ast \Omega \; \right|_{t=0} \; = \;   \left. \int_{M_j} \frac{d}{dt} \phi_t^\ast \Omega \; \right|_{t=0},
\end{equation}
where the second equality is the change of variables formula and  the last integrand is the Lie derivative of $\Omega$  with respect to $\chi$. According to  Cartan Formula \cite[Proposition 5.3.1]{Lang99} and the fact that $\Omega$ is closed, the Lie derivative of $\Omega$ is the derivative $d(\chi \lrcorner \Omega)$ of the contraction of $\Omega$ with $\chi$. Using the relation
\[\partial_z f(z,w) \cdot dz + \partial_w f(z,w) \cdot dw \; =\; 0 \]
valid on $C$ and the relation \eqref{eq:relationchi} specialized at $t=0$, we obtain that
\begin{equation}\label{eq:contraction}
(\chi \lrcorner \Omega)_{\vert_C} \; =\;  \frac{\alpha\cdot dw - \beta \cdot  dz}{zw} \; =\; \frac{-dz}{zw}\Big( \frac{\partial_z f}{\partial_w f} \alpha  +\beta\Big) \; =\; -\dfrac{\alpha \cdot \partial_z f + \beta \cdot \partial_w f}{zw\partial_w f} dz \; =\; \frac{-h}{zw\partial_w f} dz,
\end{equation}
whose pullback to $\widetilde C$ is the holomorphic form $-\xi_h$, see Lemma \ref{lem:1form}. Putting everything together, we obtain that
\begin{equation}\label{eq:diffzero}
0  \; = \;   \left. \int_{M_j}\frac{d}{dt} \phi_t^\ast \Omega \; \right|_{t=0} \; = \; \int_{M_j}  d(\chi \lrcorner\Omega) \; = \; \int_{\partial M_j}  \chi \lrcorner\Omega \; = \; -\int_{\gamma_j}  \xi_h
\end{equation}
where the penultimate equality follows from Stokes' Theorem and the last equality from \eqref{eq:contraction}. Since \eqref{eq:diffzero} is valid for any $j$, and since the homology classes $[\gamma_j]$ form a basis of $H_1(\widetilde C,\Z)$, we deduce that $\xi_h$ is zero, and in turn, that $h=0$. This is a contradiction.

It remains to prove the existence of the vector field $\chi$. As a first approximation, we can define
$\chi$ to be the meromorphic vector field defined by
\[ \alpha := \frac{\mu h}{\partial_z f-t\partial_z h} + \lambda(\partial_w f-t\partial_w h) \quad \text{and} \quad
\beta := \frac{(1-\mu) h}{\partial_w f-t\partial_w h} - \lambda(\partial_z f-t\partial_z h) \]
with $\mu, \lambda \in \C$. Since $\varepsilon$ is arbitrarily small, each curve $C_t\cap V_j$ is
smooth. In particular, the gradient $(\partial_z f-t\partial_z h,\partial_w f-t\partial_w h)$ of $f-th$
is nowhere vanishing on $\mathscr{C}:=\bigcup_{\vert t \vert<\varepsilon} \; C_t\cap V_j$. Therefore,
setting $\mu$ to either $0$ or $1$ and $\vert \lambda \vert$ large enough, we can ensure that $
\chi$ is holomorphic and nowhere vanishing on $\mathscr{C}$. Since $\chi$ satisfies \eqref{eq:relationchi}, the flow $\phi_t$ maps $C_0$ to $C_t$.
A priori, we cannot guarantee that $\chi$ is well defined on the whole $V_j\times \left\lbrace \vert t \vert <\varepsilon \right\rbrace$, since it is only meromorphic and may have poles inside the latter set. To fix this, we can replace $\chi$ with $\Psi\cdot \chi$, where $\Psi(z,w,t)$ is a test function supported on an arbitrarily small neighborhood of $\mathscr{C}$ and such that $\Psi_{\vert\mathscr{C}}=1$. The resulting vector field is $C^\infty$ and has the required properties.
\end{proof}

\section{Simple Harnack curves}\label{sec:harnack}

In this section, we label the edges of $\Delta$ by $\Delta_1,\dotsc, \Delta_n$ according to the counterclockwise ordering on $\partial \Delta$. For an edge $\Delta_j$, we denote by $\D_j$ the corresponding toric divisor in $X$. We also denote by $(z,w)$ the coordinates on $\Xb \simeq \ttor$ induced by the dual lattice $\Z^2$.
The complex conjugation on $\ttor$ extends to an anti-holomorphic involution $\conj$ on $X$. A curve $[C] \in \LL$ is real if $\conj(C)=C$ and we denote by $\R C$ the fixed locus of $\conj_{\vert C}$. For convenience, denote $C^\bullet:=C\cap\Xb$. Recall the \emph{amoeba map}
\[
\begin{array}{rcl}
\A\: \ttor & \rightarrow & \R^2\\
(z,w) & \mapsto & \big( \log \vert z \vert, \log \vert w \vert \big)

\end{array}.
\]

\begin{defn}\label{def:harnack}
A real curve $[C] \in \LL$ is a (possibly singular) \emph{simple Harnack curve} if the restriction of the amoeba map $\A\: C^\bullet \rightarrow \R^2$ is at most $2$-to-$1$.
\end{defn}

The above definition is shown to be equivalent to the original definition \cite[Definitions $2$ and $3$]{MR} in \cite[Theorem $1$]{MR}. It is clear from the definition of the map $\A$ that it identifies pairs of complex conjugated points on a real curve $C$. Therefore, simple Harnack curves are exactly those curves $C$ for which the restriction of $\A$ to $C^\bullet$ realizes the quotient of $C^\bullet$ by the involution $\conj$, see e.g., \cite[Lemma 5]{MR}. In particular, the restriction of $\A$ to $\R C^\bullet$ is $1$-to-$1$ onto the boundary of $\A(C^\bullet)$, see also \cite[Lemma 8 and Corollary 9]{Mikh} for smooth curves.

Recall that any smooth simple Harnack curve $C$ satisfies the following properties: $C$ is maximal, i.e., $b_0(\R C)=g+1$, where $g$ is the genus of $C$, and the real part $\R C$ has a unique component that intersects the boundary divisor $\partial X$. It follows that $\A(C^\bullet)$ is a topological disc with $g$ holes bounded by the images of the $g$ remaining components of $\R C$ and with punctures on its boundary. Smooth/singular simple Harnack curves are obtained from each other by contracting/expanding the latter holes, as shown by the following lemma.

\begin{lem}\label{lem:singharnack}
The only singularities of simple Harnack curves are real isolated double points. Moreover, there exists a smooth simple Harnack curve in the neighborhood of any singular simple Harnack curve $[C] \in \LL$.
\end{lem}

\begin{proof}
The first property is actually a part of the original definition \cite[Definition $3$]{MR}. With respect to Definition \ref{def:harnack} above, the latter property corresponds to \cite[Lemma $6$]{MR}. For the second part of the statement, let us fix a defining Laurent polynomial $f(z,w)$ for $C^\bullet$. Then, there exists another Laurent polynomial $h(z,w)$ satisfying the following:

-- the Newton polygon $\New(h)$ of $h$ equals $\New(f)=:\Delta$;

-- for any node $(z,w) \in C^\bullet$, we have $h(z,w)=-1$ (resp. $h(z,w)=1$) if $f$ is positive (resp. negative) in a small neighborhood $\mathcal{U} \subset (\R^*)^2$ of $(z,w)$.\\
Indeed, the latter conditions are linear in the space $\LL$ and there are as many conditions as nodes. In particular, there are at most $\vert \Delta^\circ \cap \Z^2\vert $ of them by Khovanskii's formula, which in turn is strictly less than the dimension of $\LL$. The space of polynomials $h$ as above is therefore not empty.

It follows now from the Morse Lemma that the curve defined by $f+\varepsilon h$ is a smooth simple Harnack curve for any arbitrarily small $\varepsilon>0$. Indeed, for any node $(z,w)$, there exist real analytic coordinates centered at $(z,w)$ such that $f+\varepsilon h$ reads either $z^2+w^2-\varepsilon$ or $\varepsilon-z^2-w^2$, that is the isolated double point of $C$ are deformed into compact ovals in $\R C^\bullet$. The resulting curve satisfies therefore the requirements of \cite[Definition $2$]{MR} which is equivalent to Definition \ref{def:harnack}.
\end{proof}

Recall that for smooth simple Harnack curve $C$, the \emph{order map} $\ord$ of \cite{FPT} establishes a bijective correspondence between the set of compact connected components of $\R C^\bullet$ and the set of lattice points $\Delta^\circ \cap \Z^2$, see \cite[Corollary $10$]{Mikh}.

According to \cite[Lemma 11]{Mikh}, the map $\ord$ on a smooth simple Harnack curve $C$ can
be described as follows. First, assume that the vertex $\Delta_{n}\cap\Delta_{1}$ is the origin. Let
$c_0\subset\R C^\bullet$ be the unique connected component joining the two consecutive toric
divisors $\D_{n}$ and $\D_{1}$. For any compact component $c$ of $\R C^\bullet$, draw a path $
\rho_c \subset \A(C)$ joining $\A(c_0)$ to $\A(c)$. By the $2$-to-$1$ property of the amoeba
map, the lift $\gamma_c:=\A^{-1}(\rho_c)$ is a loop in $C^\bullet$, which is invariant under complex conjugation. There exists a unique orientation of $\gamma_c$ such that the corresponding homology class $(a,b) \in H_1(\ttor, \Z)$ satisfies $(-b,a) \in \Delta^\circ \cap \Z^2$ (note the sign mistake in the sixth line of the proof of \cite[Lemma 11]{Mikh}). Then, we have $\ord(c)=(-b,a)$. Observe that the homology class of $\gamma_c$ is independent of the choice of $\rho_c$. In particular, the map $\ord$ is well defined.

For a singular simple Harnack curve $C$, a compact component $c$ of $\R C^\bullet$ can be of two types: either $c$ is a topological circle or it is a real isolated double point of $C$. In the former case, we define $\gamma_c$ as above. If $c$ is now a node of $C$  we can repeat the same construction as above where $\rho_c$ is a path joining $\A(c_0)$ to $\A(c)$ and $\gamma_c:=\A^{-1}(\rho_c)$.

\begin{defn}\label{def:ord}
For a singular simple Harnack curve $C$, define the \emph{order map}

$$\ord\: \{\text{compact components of } \R C^\bullet\} \rightarrow \Delta^\circ \cap \Z^2$$
by $\ord(\nu)=(-b,a)$ where $(a,b) \in H_1(\Xb, \Z)$ is the homology class of the (properly oriented) loop $\gamma_c$ constructed above.
\end{defn}

The fact that  $(-b,a) \in \Delta^\circ \cap \Z^2$ (with the appropriate orientation of $\gamma_c$) follows from the facts that

-- it holds for smooth curves $C$;

-- the homology class of $\gamma_c$ is constant under small perturbations of $C$;

-- singular Harnack curves can always be perturbed into smooth ones, thanks to Lemma \ref{lem:singharnack}.\\
By the same arguments, the statement below follows now from \cite[Corollary $10$]{Mikh}.

\begin{prop}\label{prop:extord}
For any singular simple Harnack curve, the order map is a bijection.
\end{prop}

The existence of smooth simple Harnack curves in $\LL$ is guaranteed by \cite[Corollary A4]{Mikh}. For singular ones, the existence is addressed in \cite[Theorem 6]{KO}, \cite[Theorems 2 and 10]{Bru14},  \cite[Theorem 3]{CL1} in various contexts. The theorem below discusses the existence of singular Harnack curves with prescribed order map.

\begin{thm}\label{thm:contraction}
For any integer $0 < g < \vert \Delta^\circ \cap \Z^2\vert$ and any subset $E \subset \Delta^\circ \cap \Z^2$ of cardinality $\delta:=|\Delta^\circ \cap \Z^2\vert-g$,
there exists a one-parameter family of simple Harnack curves $\{[C_t]\}_{t\in [0,1]}\subset |\Li|$ such that
$C_t$ is smooth for all $t<1$, the curve $C_1$ is singular and $\ord(\{\text{nodes of }C_1\})=E$. In particular, any Severi variety $\sv{\Li, g}^{\irr}$ contains a simple Harnack curve.
\end{thm}

In \cite{KO}, it was observed that the Euclidean area of the $g$ compact holes of the amoeba $\A(C^\bullet)$ of a smooth simple Harnack curve can be completed into a system of coordinates on the space of all simple Harnack curves for $X=\mathbb{P}^2$, see \cite[Theorem 6]{KO}. In particular, any sub-collection of these holes can be contracted to points by prescribing the corresponding area to tend to zero.

The proof of Theorem \ref{thm:contraction} is based on the same idea. Below, we give an alternative proof of the fact that the area of the holes of $\A(\Cb)$ provides local coordinates, building on the material of Section \ref{sec:def}.

\begin{proof}[Proof of Theorem \ref{thm:contraction}]
Choose any smooth simple Harnack curve $[C_0] \in \LL$, and denote by $H$ the space of smooth simple Harnack curves in $\LL$ that coincide with $C_0$ on  $X \setminus \Xb$. Note that $H \subset \LL$ is a smooth subvariety of real dimension $g_\Delta:=\vert \Delta^\circ \cap \Z^2\vert$.

Consider the continuous map $\Area\: H \rightarrow \R_{\geqslant 0}^{\Delta^\circ \cap \Z^2}$ that associates to any curve $C$ the Euclidean area of the holes of $\A(C^\bullet)$, where each hole is indexed by the corresponding point in $\Delta^\circ \cap \Z^2$ via the order map $\ord$. We claim that the map $\Area$ is a local diffeomorphism. Indeed, the map $\Area$ is nothing but the map $\Phi_\Upsilon$, where $\Upsilon$ consists of the $g_\Delta$ compact connected components of $\R \Cb$ (note that the existence of the map $\ord$ implies that the monodromy of $H$ acts trivially on $\Upsilon$). To see this, pick membranes $M_j$ to be the disc in $(\R^*)^2$ bounded by $\gamma_j$. Since the $2$-form $\Omega$ is the pullback under the coordinate-wise complex logarithm of the Euclidean form $dz\wedge dw$, the integrals \eqref{eq:integrals1} compute the Euclidean area of the holes of $\A(\Cb)$, and the claim follows by Theorem~\ref{thm:integrals1}.

Next, notice that the map $\Area$ extends continuously to the closure $\overline{H}$. Indeed, by Lemma \ref{lem:singharnack}, the curves in the boundary of $H$ correspond to the singular curves in which some of the $g_\Delta$ real ovals are contracted to real isolated double points. Therefore, we obtain a continuous map $\Area\: \overline{H} \rightarrow \R_{\geqslant 0}^{\Delta^\circ \cap \Z^2}$. Moreover, by \cite[Theorem 6]{KO}, the latter map is proper in the case of planar curves, and the argument in {\em loc.cit.} generalizes to arbitrary toric surfaces verbatim.

Finally, set $(a_j)_{j\in \Delta^\circ \cap \Z^2}:=\Area([C_0])$, and consider the segment $\sigma \subset \R^{\Delta^\circ \cap \Z^2}$ from the point $\Area([C_0])$ to the point $(b_j)_{j\in \Delta^\circ \cap \Z^2}$ defined by $b_j =0$ if  $j \in E$,  and  $b_j =a_j$ otherwise. We claim that there exists a path in $\overline{H}$ starting from $[C_0]$ and mapping bijectively onto $\sigma$. To see this, define $\mathscr{P}$ to be the set of all closed paths $\rho$  starting from $[C_0]$ and mapping injectively into $\sigma$. The set $\mathscr{P}$ is totally ordered and contains paths with non-empty relative interior since $\Area$ is a covering. Since $\Area$ is proper on $\overline{H}$, there exists a unique maximal element $\tilde{\rho}$ in $\mathscr{P}$. Assume towards the contradiction that the end point $[\widetilde C]$ of $\tilde \rho$ maps to an interior point of $\sigma$. Then $\A(\widetilde{C}^\bullet)$ bounds $g_\Delta$ compact holes, or equivalently $\R \widetilde{C}^\bullet$ contains $g_\Delta$ ovals, that is $\widetilde C$ is smooth. Since  $\Area$ is a local covering on $H$, the smoothness of $\widetilde C$ implies that $\tilde \rho$ can be extended into a longer path in $\mathscr{P}$. This is a contradiction. Therefore $\tilde \rho$ maps bijectively to $\sigma$ and provides the sought degeneration of $C_0$ to $C_1$.
\end{proof}

\section{Topological proof of the main results}\label{sec:proofs}
In this section, we use the notation $\Delta_j$, $\D_j$ introduced in the previous section.
We assume with no loss of generality that $\Delta_n\cap\Delta_1=0$.

\begin{lem}\label{lem:harnackGphi}
Let $[C]\in \sv{g,\Li}^{\irr}$ be a torically transverse simple Harnack curve and denote $E=\ord(\{\text{nodes of }C\})$. If $V\subset \sv{g,\Li}^{\irr}$ is the irreducible component containing $[C]$, then the lattice $M(V)$ of Section \ref{sec:toplat} is generated by $(\Delta \cap \Z^2) \setminus E$.
\end{lem}

\begin{proof}
As in Section \ref{sec:harnack}, we denote by $c_0$ the unique connected component of $\R \Cb$ joining $\D_n$ to $\D_1$. We denote by $a_1,\dotsc, a_g$ the $g$ ovals of $\R \Cb$. For any $a_j$, we choose a path inside $\A(\Cb)$ joining $\A(c_0)$ to $\A(a_j)$, and denote by $a_{g+j}\subset \Cb$ its preimage under the map $\A$. By perturbing the paths if necessary, we can make sure that the loops $a_{g+j}$ avoid the nodes of $C$. Finally, let $\ell_1,\dotsc, \ell_k$ be small loops around the punctures of $\Cb$. Then, the collection of simple closed curves $\ell_1,
\dotsc,\ell_k$, $a_1,\dotsc,a_{2g}$ (oriented arbitrarily) lifts to a basis of  $H_1(\tCb,\Z)$.
In particular, the homology classes of these curves generate the sublattice $N_C$.

Each $\ell_j$ is associated to a puncture
of $\Cb$ corresponding to the intersection of $C$ with one of the toric divisors of $X$. In turn,
such toric divisor corresponds to an edge of $\Delta$ directed by some primitive integer vector $
(a,b)$. It is a standard fact in toric geometry that the homology class $[\ell_j]$ is then given by $
\pm(-b,a)$ (depending on the orientation of $\ell_j$), cf. \cite[Lemma~ $1.10$]{L2}. In
particular, the sublattice of $N$ generated by the $[\ell_j]$'s is the rotation by $\frac{\pi}{2}$ of the
sublattice $\left\langle \partial \Delta \cap \Z^2\right\rangle$. By the construction and the definition of the order map, the following holds: for any $1\leqslant j \leqslant g$ we have $[a_j]=(0,0)$ and $[a_{g+j}]=\pm(-b,a)$, where $\ord(a_j)=(a,b)$. We deduce that $N_C$ is the
rotation by $\frac{\pi}{2}$ of the lattice generated by $(\Delta \cap \Z^2) \setminus E$, which implies the assertion.
\end{proof}

\begin{proof}[Proof of Theorem \ref{thm:general}]
The theorem follows from Lemmas \ref{lem:topinv} and \ref{lem:harnackGphi} and Theorem \ref{thm:contraction}.
\end{proof}

\begin{lem}\label{lem:kite}
Let $\Delta_{k,k'}\subset \R^2$ be a kite, $g,\delta$ integers such that $\delta+g=k+k'-1$, and $[C]\in \sv{g,\Li}^{\irr}$ a simple Harnack curve. Set $E:= \ord(\{\text{nodes of } C\})$ and
$$E^{even}:= E\cap (\{0\}\times 2\Z) \; \; \text{and} \; \; E^{odd}:=E\setminus E^{even}.$$
Then, the nodal partition of $C$ is $\{\vert E^{even}\vert, \, \vert E^{odd}\vert \}$. In particular,
any integer partition $\delta=\delta_1+\delta_2$ such that $0 \leqslant \delta_1, \delta_2 \leqslant \lceil \frac{k+k'-1}{2} \rceil$  is the nodal partition of some curve in $\sv{g,\Li}^{\irr}$.
\end{lem}

\begin{proof}
By definition, the nodal partition of $C$ is the partition given by the sign of the first coordinate of the nodes. By\cite[Lemma~11]{Mikh}, the sign corresponding to a node $\nu\in C$ is given by the parity of the second coordinate of $\ord(\nu)$. This proves the first part of the statement. In particular, the nodal partition $\delta=\delta_1+\delta_2$ of $C$ has to satisfy $0 \leqslant \delta_1, \delta_2 \leqslant \lceil \frac{k+k'-1}{2}  \rceil$, and for each such partition we can find $E$ such that $\{\vert E^{even}\vert, \, \vert E^{odd}\vert \}=\{\delta_1,\delta_2\}$. The result now follows from Theorem \ref{thm:contraction}.
\end{proof}

\begin{proof}[Proof of Theorem \ref{thm:kite}]
As in the tropical proof of the theorem, it is sufficient to construct for any admissible pair $(M,\kappa)$, an irreducible component $V\subseteq V_{g,\Li}^\irr$ such that $(M(V),\kappa(V))=(M,\kappa)$; see \S\ref{sec:proofkite} for the definition  of an admissible pair. By Lemmata~\ref{lem:harnackGphi} and \ref{lem:kite}, it suffices to find a subset $E\subseteq \Delta^\circ\cap M$ of cardinality $\delta$ such that $(\Delta\cap M)\setminus E$ generates $M$, and $\left||E^{even}|-|E^{odd}|\right|=\kappa$. Such a set $E$ can be taken as the set of interior vertices of the $M$-integral triangulation of $\Delta$ constructed in Section \ref{sec:proofkite}.
\end{proof}

\appendix
\section{}

In this section, we establish a connection between the Severi varieties associated to kites and the topological classification of polynomials as studied in \cite{KZ} and \cite{Zv}.

Let $\Delta:=\Delta_{k',k}$ be a kite, and $(X,\Li)$ the corresponding polarized toric surface. In this section it will be convenient to shift the kite so that its vertices become $(\pm 1,0),(0,-k),$ and $(0,k')$. For a given integer $g\geqslant 0$, set $\delta:=\delta_{\Z^2}(\Delta,g)$. Recall that a curve $[C]\in \LL$ is defined by a Laurent polynomial of the form
\begin{equation}
f(z,w)= \frac{a}{z}+p(w)+bz
\label{eq:defpol}
\end{equation}
where $a, b\in \C$ and $p(w)$ is a univariate Laurent polynomial. Recall from Section \ref{sec:nodalpart} that any curve $[C]\in V_{g,\Li}^{\irr}$ has an associated nodal partition $\delta=\delta_1+\delta_2$. Our first observation is that the later partition determines the \emph{passport} of the polynomial $p(w)$.

Consider a branched cover $p : S \rightarrow \cp{1}$ of degree $d$ from a compact orientable surface $S$. For any critical value $b\in \cp{1}$ of $p$ with fiber $p^{-1}(b):=\{a_1,\dotsc,a_k\}$, we have a partition $d=d_1+\cdots+d_k$ where $d_i$ is the multiplicity of $p$ at $a_i$. In turn,  the passport $\Sigma$ of $p$ is the collection of the partitions for all critical values of $p$ other than $\infty$. In particular, any Laurent polynomial admits a passport.

Denote by $\pkk$ the space of Laurent polynomials with set of exponents contained in $[-k,k']$, to which the polynomial $p(w)$ from \eqref{eq:defpol} belongs. For a given passport $\Sigma$, denote by $\pkk_\Sigma$ the subspace of $\pkk$ consisting of polynomials with passport $\Sigma$. For $k=0$, the space $\mathcal{P}^{k',k}_\Sigma$ is a central object in the topological classification of polynomials, see for instance \cite{KZ}. It follows from the proof of \cite[Theorem 6]{KZ} that $\pkk_\Sigma$ is a smooth variety of complex dimension $m+2$, where $m$ is the number of distinct finite critical values of $\Sigma$. In particular, a path-connected component of $\pkk_\Sigma$ is irreducible and vice versa.

Let us denote by $\{2^\delta\}$ the partition of $d$ with exactly $\delta$ summands ``2'' and remaining summands ``1''. Since the case $\delta=0$ corresponds to a regular value, every time $\{2^0\}$ appears in a passport we disregard it. Then, we have the following simple lemma.

\begin{lem}\label{lem:passport}
Let $V\subseteq V_{g,\Li}^{\irr}$ be an irreducible component with nodal partition $\delta=\delta_1+\delta_2$, then there is an open Zariski subset of curves $[C]\in V$ for which
the Laurent polynomial $p(w)$ of \eqref{eq:defpol} has passport $\big\lbrace \{ 2^{\delta_1} \}, \{ 2^{\delta_2} \}, \{2\}, \cdots,\{2\} \big\rbrace$.
\end{lem}

\begin{proof}
Assume for simplicity that $0<\delta_2\leqslant\delta_1$. The case $\delta_2=0$ is similar.
The coordinates $(z,w)$ of a node of a curve $[C]\in V_{g,\Li}^{\irr}$ defined by \eqref{eq:defpol} satisfy the conditions $f(z,w)=\partial_zf(z,w)=\partial_wf(z,w)$ or equivalently
\[p'(w)=0, \; z^2=\frac{a}{b} \; \text{ and } \; p(w)=-\left(\frac{a}{z}+bz\right).\]
It follows that $p(w)=\pm2\sqrt{ab}$, and the nodal partition corresponds to the distribution of the $w$'s over the two critical values $\pm2\sqrt{ab}$ among the $\delta$-many nodes of $C$. In particular, the number $m$ of distinct finite critical values of the polynomial $p(w)$ is at most $k+k'-1 -(\delta-2)=k+k'+1-\delta$ and there is equality if and only if the passport of $p(w)$ is $\big\lbrace \{ 2^{\delta_1} \}, \{ 2^{\delta_2} \}, \{2\}, \cdots,\{2\} \big\rbrace$.

Clearly, the passport of $p(w)$ is constant as long as $[C]$ lies in a certain open Zariski subset of $V$. We denote this generic passport by $\Sigma$. According to the dimension formula of $\pkk_\Sigma$ and the fact that the $2$ critical values of $p(w)$ corresponding to nodes of $C$ are determined by $a$ and $b$, the locus of polynomials $p(w)$ corresponding to general curves $[C]\in V$ has dimension $m+2-2=m$. Therefore, $$k+k'+2-\delta=\dim(V)=m+2-1=m+1,$$
and hence $m=k+k'+1-\delta$. Thus, $\Sigma$ is the sought passport.
\end{proof}

\begin{lem}\label{lem:irredcompbij}
Let $\Sigma:= \big\lbrace \{ 2^{\delta_1} \}, \{ 2^{\delta_2} \}, \{2\}, \cdots,\{2\} \big\rbrace$ be a passport, where $\delta=\delta_1+\delta_2$ is a partition of $\delta$. Then, the number of irreducible components $V\subseteq  V_{g,\Li}^{\irr}$ with nodal partition $\delta=\delta_1+\delta_2$ is equal to the number of irreducible components of $\pkk_\Sigma$.
\end{lem}

\begin{proof}
Again, we assume for simplicity that $0<\delta_2\leqslant\delta_1$. The case $\delta_2=0$ is similar. Denote by $V_{\{\delta_1,\delta_2\}} \subseteq V_{g,\Li}^{\irr}$ the union of the irreducible components with nodal partition $\delta=\delta_1+\delta_2$, and by $\mathscr{C}_{\{\delta_1,\delta_2\}}\subset H^0(X,\Li)\setminus\{0\}$ the cone over it. Then there is a natural bijection between the sets of the irreducible components of $V_{\{\delta_1,\delta_2\}}$ and $\mathscr{C}_{\{\delta_1,\delta_2\}}$. Furthermore, $V_{\{\delta_1,\delta_2\}}$ and $\mathscr{C}_{\{\delta_1,\delta_2\}}$ are disjoint unions of their irreducible components.

By Lemma~\ref{lem:passport}, there exists a natural map $g\: \mathscr{C}_{\{\delta_1,\delta_2\}}\rightarrow\pkk_\Sigma$ mapping $f(z,w)= \frac{a}{z}+p(w)+bz$ to $p(w)$. Plainly, $g$ is continuous, and hence induces an injective map between the sets of irreducible components. We claim that the latter map is surjective. Indeed, any component $W\subseteq \pkk_\Sigma$ is invariant under affine transformations $p\mapsto\alpha p+\beta$, where $\alpha\in \C^\ast, \beta \in \C$. Therefore, $W$ contains $p$, whose two special critical values are $\pm 2$. Let $C$ be the curve given by the section $f(z,w)=\frac{1}{z}+p(w)+z$. It has nodal partition $\delta=\delta_1+\delta_2$, an hence $W$ contains the image of the irreducible component of $\mathscr{C}_{\{\delta_1,\delta_2\}}$ containing $f$. This completes the proof.
\end{proof}

The above lemma can be used to obtain information on Severi varieties from the spaces $\pkk_\Sigma$ or the other way around, as illustrated by the following.

\begin{cor}\label{lem:delta2}
Assume that $k+k'\geqslant 5$ and that $\delta=2$. Then, the Severi variety $V_{g,\Li}^{\irr}$ has exactly $2$ irreducible components.
\end{cor}

\begin{proof}
Since $\delta=2$, there are exactly 2 possible nodal partitions, namely $2=1+1$ and $2=2+0$. If $\Sigma_1$ and $\Sigma_2$ are the corresponding passport as in Lemma \ref{lem:irredcompbij}, then $\pkk_{\Sigma_1}$ is the space of Laurent polynomials with distinct critical values. In particular, it is irreducible. The polynomials in $\pkk_{\Sigma_2}$ have only simple critical values except one, whose preimage contains $2$ critical point. We claim that $\pkk_{\Sigma_2}$ is also irreducible. Indeed, acting by pre-composition with an affine linear map, we can fix the $2$ critical points over the special critical value to be $0$ and $1$. As the polynomials in $\pkk_{\Sigma_2}$ satisfying this extra condition form a linear subspace, this proves the claim. The result now follows from Lemma \ref{lem:irredcompbij}.
\end{proof}

\begin{cor}\label{cor:lowerbound}
Let $k,k',g,\delta,\kappa,\delta_1,$ and $\delta_2$ be non-negative integers such that $k'\geqslant k$,  $k'>0$, $g+\delta=k+k'-1$, $\delta=\delta_1+\delta_2$ and $\kappa=\vert \delta_1-\delta_2\vert$. Denote by $\Sigma:=\big\lbrace \{ 2^{\delta_1} \}, \{ 2^{\delta_2} \}, \{2\}, \cdots,\{2\} \big\rbrace$  the passport of Laurent polynomials in $\pkk$. Then, the number of irreducible components of $\pkk_\Sigma$ is bounded from below by the number of lattices $M\subset \Z^2$ such that $(M,\kappa)$ is an admissible pair, see Section \ref{sec:proofkite}.
\end{cor}

\begin{proof}
By Lemma \ref{lem:irredcompbij}, the number of irreducible components of $\pkk_\Sigma$ is equal to the number of irreducible components of $V^{\irr}_{g,\Li}$ with nodal partition $\delta=\delta_1+\delta_2$ or equivalently with signature $\kappa$. According to Section \ref{sec:proofkite}, the later number is bounded from below by the number of lattices $M\subset \Z^2$ such that $(M,\kappa)$ is an admissible pair.
\end{proof}

\begin{rem}
(i) For a sublattice $M\subseteq \Z^2$ of index $r$, a component $V\subseteq V_{g,\Li}^{\irr}$ for which $M(V)=M$, and any curve $[C]\in V$, the Laurent polynomial $p(w)$ of \eqref{eq:defpol} can be written $p(w)=q(w)^r$ if $r$ is odd, and $p(w)=q(w)^r-2(ab)^{r/2}$ if $r$ is even. Therefore, the existence of lattice $M$ compatible with $\kappa$ and with index $r\geqslant 2$ implies that the passport $\Sigma$ is decomposable in the sense of \cite{Zv}. %This is a valuable information since there seems not to exist any criterion to determine the decomposability of a passport.

(ii) Let $M, V, r, p,$ and $q$ be as above. By \cite{CHT20a}, the closure of any component of the Severi variety contains simple rational Harnak curves, and therefore $V$ contains simple Harnack curves. For simple Harnak curves, an explicit computation shows that the polynomial $q(w)$ is of type $S_n$ in the sense of \cite[Section 4.1]{Zv}, with $n=k+k'$. Thus, \cite[Conjecture~13]{Zv}, if true, implies the sharpness of the bound $\#_{k,k',g}$ of Theorem \ref{thm:kite} in the case $k=0$.

(iii) One could study the decomposability of more general passports by considering non-complete linear systems associated to support sets $A\subset \Z^2$, following the $A$-philosophy of \cite{GKZ}. For instance, let $\Delta$ be the polygon with vertices $(-m,0),(1,0),(0,k'), (0,-k)$ for some $k'\geqslant k\geqslant 0$ and $m\geqslant 2$, and $A$ the set consisting of the vertices of $\Delta$ and of the inner integral points of $\Delta$ that belong to the $y$-axis. Consider the associated polarized toric variety $(X,\Li)$ and the non-complete linear system $|\Li_A|\subset |\Li|$. Then the curves in $\vert \Li_A\vert$ are given by polynomials of the form
\[f(z,w)=\frac{a}{z^m}+p(w)+bz.\]
Such curves admit nodal partitions of length $m$. Thus, in this case the Laurent polynomials $p(w)$ may have up to $m+1$ special critical values, instead of $2$. It would be interesting to study the decomposability of the corresponding passports via the reducibility of the Severi varieties $V^{\irr}_{g,\Li_A}$.
\end{rem}

\bibliographystyle{alpha}
\bibliography{Draft}
\end{document}